\pdfoutput=1
\documentclass{article}
\usepackage[parfill]{parskip}
\usepackage{amsmath, amsthm, amssymb}
\usepackage{eucal}
\usepackage{mathrsfs}
\usepackage{geometry}
\usepackage[unicode,bookmarksnumbered=true,colorlinks=true,allcolors=blue,linktoc=all]{hyperref}
\usepackage[capitalise,nameinlink]{cleveref}
\usepackage{tikz-cd}
\usepackage{stmaryrd}
\usepackage{listings}
\usepackage{etoolbox}
\usepackage{thm-restate}
\usepackage[nottoc,notlot,notlof]{tocbibind}
\usepackage{enumitem}
\usepackage{mathtools}
\usepackage{quiver}
\usepackage{caption}

\crefname{subsection}{Subsection}{subsections}

\hypersetup{bookmarksdepth=2}
\newtheorem{theorem}{Theorem}

\newtheorem{thm}{Theorem}[section]
\newtheorem{lem}[thm]{Lemma}
\newtheorem{prop}[thm]{Proposition}

\theoremstyle{definition}
\newtheorem{defn}[thm]{Definition}
\newtheorem{notn}[thm]{Notation}

\newtheorem{remark}[thm]{Remark}

\newcommand{\ncmd}{\newcommand}

\definecolor{DefColor}{rgb}{0.6,0.15,0.25}
\newcommand{\mdef}[1]{\textcolor{DefColor}{#1}}
\newcommand{\tdef}[1]{\mdef{\emph{#1}}}

\ncmd{\mbb}[1]{\mathbb{#1}}
\ncmd{\mrm}[1]{\mathrm{#1}}
\ncmd{\mcl}[1]{\mathcal{#1}}
\ncmd{\mfk}[1]{\mathfrak{#1}}
\ncmd{\mbf}[1]{\mathbf{#1}}
\ncmd{\mscr}[1]{\mathscr{#1}}

\ncmd{\todo}[1]{\textbf{TODO #1}}
\ncmd{\reftodo}[1]{\textbf{REF #1}}

\DeclareRobustCommand{\minwidthbox}[2]{%
  \mathmakebox[\ifdim#2<\width\width\else#2\fi]{#1}%
}
\ncmd{\too}[1][]{\xrightarrow{\minwidthbox{#1}{1em}}}
\ncmd{\iso}{\too[\smash{\raisebox{-0.5ex}{\ensuremath{\scriptstyle\sim}}}]}
\ncmd{\oot}[1][]{\xleftarrow{\minwidthbox{#1}{1em}}}
\ncmd{\hooktoo}[1][]{\xhookrightarrow{\minwidthbox{#1}{1em}}}
\ncmd{\adj}[1][]{\mathrel{\substack{\xrightarrow{\minwidthbox{#1}{1em}} \\[-.7ex] \xleftarrow{\minwidthbox{#1}{1em}}}}}

\ncmd{\qin}{\quad\in\quad}

\ncmd{\Id}{\mrm{Id}}
\ncmd{\tId}{\mbf{Id}}
\ncmd{\Nm}{\mathrm{Nm}}

\ncmd{\BB}{\mrm{B}}
\ncmd{\CC}{\mcl{C}}
\ncmd{\DD}{\mcl{D}}
\ncmd{\EE}{\mcl{E}}
\ncmd{\tA}{\mbf{A}}
\ncmd{\trA}{\tA}
\ncmd{\tB}{\mbf{B}}
\ncmd{\cK}{\mcl{K}}
\ncmd{\cX}{\mcl{X}}
\ncmd{\lsF}{\mscr{F}}
\ncmd{\lsG}{\mscr{G}}

\ncmd{\LL}{\mrm{L}}
\ncmd{\RR}{\mrm{R}}
\ncmd{\BC}{\mrm{BC}}
\ncmd{\op}{\mrm{op}}
\ncmd{\co}{2\text{-}\mrm{op}}
\ncmd{\trop}{3\text{-}\mrm{op}}

\ncmd{\pt}{\mrm{pt}}
\ncmd{\Mod}{\mrm{Mod}}
\ncmd{\Spaces}{\mcl{S}}
\ncmd{\SpacesK}{\Spaces_\cK}
\ncmd{\Span}{\mrm{Span}}
\ncmd{\tSpan}{\mbf{Span}}
\ncmd{\thSpan}{\tSpan}
\ncmd{\Spano}{\Span_1}
\ncmd{\Spanoh}{\tSpan_{1\!\frac{1}{2}}}
\ncmd{\Spant}{\tSpan_2}
\ncmd{\Spanth}{\thSpan_{2\!\frac{1}{2}}}
\ncmd{\Spantr}{\thSpan_{3}}
\ncmd{\Cat}{\mrm{Cat}}
\ncmd{\CatK}{\Cat_\cK}
\ncmd{\CatKlim}{\Cat^\cK}
\ncmd{\Catmfin}{\Cat_{m\text{-}\mrm{fin}}}
\ncmd{\CAT}{\widehat{\Cat}}
\ncmd{\CATall}{\CAT_{\mrm{all}}}
\ncmd{\Spacesmfin}{\Spaces_{m\text{-}\mrm{fin}}}
\ncmd{\tCat}{\mbf{Cat}}
\ncmd{\tCatK}{\tCat_\cK}
\ncmd{\tCatKp}{\tCat_{\cK'}}
\ncmd{\tCatKlim}{\tCat^\cK}
\ncmd{\tCatKzsa}{\tCat^{\oplus\text{-}\cK_0}}
\ncmd{\tCatmfin}{\tCat_{m\text{-}\mrm{fin}}}
\ncmd{\thCatt}{\mbf{Cat}_2}
\ncmd{\PrL}{\mrm{Pr}^{\LL}}

\ncmd{\yon}{\text{\usefont{U}{min}{m}{n}\symbol{'110}}}
\DeclareFontFamily{U}{min}{}
\DeclareFontShape{U}{min}{m}{n}{<-> dmjhira}{}

\DeclareMathOperator{\Tw}{Tw}
\DeclareMathOperator{\Map}{Map}
\DeclareMathOperator{\End}{End}
\DeclareMathOperator{\tEnd}{\mbf{End}}

\DeclareMathOperator{\Fun}{Fun}
\DeclareMathOperator{\tFun}{\mbf{Fun}}

\DeclareMathOperator*{\colim}{colim}
\DeclareMathOperator{\coliminl}{colim}
\DeclareMathOperator{\PSh}{\mcl{P}}
\DeclareMathOperator{\PShK}{\PSh^\cK}

\ncmd\noloc{%
  \nobreak
  \mspace{6mu plus 1mu}
  {:}
  \nonscript\mkern-\thinmuskip
  \mathpunct{}
  \mspace{2mu}
}

\title{Categorical Ambidexterity}
\author{Shay Ben-Moshe\thanks{Faculty of Mathematics and Computer Science, Weizmann Institute of Science, Israel.} \thanks{Max Planck Institute for Mathematics, Bonn, Germany.}}
\date{}

\begin{document}
	\maketitle
	
	\begin{abstract}
		We prove an ambidexterity result for $\infty$-categories of $\infty$-categories admitting a collection of colimits.
		This unifies and extends two known phenomena:
		the identification of limits and colimits of presentable $\infty$-categories indexed by a space,
		and the $\infty$-semiadditivity of the $\infty$-category of $\infty$-categories with $\pi$-finite colimits proven by Harpaz.
		Our proof employs Stefanich's universal property for the higher category of iterated spans, which encodes ambidexterity phenomena in a coherent fashion.
	\end{abstract}

	
	\tableofcontents

\section{Introduction}

\subsection{Main Results}

Throughout this paper, we use the terms category to mean an $\infty$-category, $2$-category to mean an $(\infty, 2)$-category, and $3$-category to mean an $(\infty, 3)$-category.

There are two important known results about ambidexterity in categories of categories:

In the category of presentable categories $\PrL$, limits and colimits indexed by (small) spaces are canonically isomorphic.
This follows from the more general fact that a limit of a diagram in $\PrL$ is computed as in $\CAT$, and the colimit is computed by passing to the right adjoint maps and computing the limit in $\CAT$, as proven by Lurie \cite[Proposition 5.5.3.13 and Theorem 5.5.3.18]{HTT} (see also \cite[Example 4.3.11]{HL}).

Let $\Catmfin$ denote the category of categories admitting $m$-finite colimits and functors preserving them.
Harpaz showed that it is $m$-semiadditive, i.e.\ that limits and colimits indexed by $m$-finite spaces are canonically isomorphic \cite[Proposition 5.26]{Harpaz}.
His proof is based on a comparison with the universal $m$-semiadditive category of spans of $m$-finite spaces $\Spano(\Spacesmfin)$.

The purpose of this paper is to unite these phenomena and provide a uniform proof.
Let $\cK_0 \subset \Spaces$ be a full subcategory of spaces including the point and closed under pullbacks, and let $\cK_0 \subset \cK \subset \Cat$ be any collection of categories containing $\cK_0$.
We let $\CatK$ denote the category of categories admitting $\cK$-indexed colimits and functors preserving them.

\begin{theorem}[{\cref{lim-colim}}]\label{lim-colim-intro}
	Let $\CC_\bullet\colon X \to \CatK$ be a diagram indexed by a space $X \in \cK_0$.
	There is a canonical equivalence
	\[
		\colim_X \CC_\bullet \iso \lim_X \CC_\bullet
		\qin \CatK.
	\]
\end{theorem}

As an example, considering large categories instead of small ones, letting $\cK_0$ be the collection of all small spaces and letting $\cK$ be the collection of all small categories, we get that limits and colimits indexed by spaces coincide already in the category of cocomplete categories $\CATall$.
This implies the same for the full subcategory $\PrL$ (see \cref{prl}).
Similarly, taking $\cK = \cK_0 = \Spacesmfin$ to be the collection of $m$-finite spaces recovers the result about $\Catmfin$.

A specific interesting case of \cref{lim-colim-intro} is when the diagram is constant on a category $\CC \in \CatK$.
In this case, we get an equivalence
\[
	\CC[X] := \colim_X \CC \iso \lim_X \CC =: \CC^X
	\qin \CatK.
\]
One may wonder about the functoriality of this equivalence in $X \in \cK_0$.
Observe that the left-hand side is canonically covariantly functorial in $X$, while the right-hand side is canonically contravariantly functorial by pre-composition.
Since $\CC$ is assumed to have $\cK$-indexed colimits, the pre-composition functoriality admits a left adjoint functoriality, given by left Kan extension.
Generalizing a result of Carmeli--Cnossen--Ramzi--Yanovski on the case of $\PrL$ from \cite[\S4.1]{CatChar}, we show that these functorialities are adjoint to one another.

\begin{theorem}[{\cref{constant}}]\label{constant-intro}
	Let $\CC \in \CatK$.
	There is a canonical equivalence
	\[
		\CC[X] \simeq \CC^X
		\qin \CatK
	\]
	natural in $X \in \cK_0$ for the left Kan extension functoriality on the right-hand side.
\end{theorem}

As one might expect, this functoriality in $X$ suitably extends to non-constant diagrams.
To encapsulate all functorialities in a coherent form, we employ several higher categorical structures.
Recall that $\CatK$ enhances to a $2$-category $\tCatK$.
Let $\thCatt$ denote the $3$-category of $2$-categories, and consider the functor
\[
	\tCatK^{(-)}\colon \cK_0^\op \too \thCatt,
\]
sending $X$ to $\tCatK^X$ and $f\colon X \to Y$ to $f^*\colon \tCatK^Y \to \tCatK^X$.
We denote by $\Spanth(\cK_0)$ the $3$-category of iterated spans, whose objects are objects of $\cK_0$, $1$-morphisms are spans, $2$-morphisms are spans between spans, and $3$-morphisms are maps of spans between spans, as depicted below.
\[\begin{tikzcd}[column sep=1.1em,row sep=small]
	&& Z &&& Z &&&& Z \\
	X & X && Y & X & Q & Y & X & Q && P & Y \\
	&&&&& W &&&& W
	\arrow[from=1-3, to=2-2]
	\arrow[from=1-3, to=2-4]
	\arrow[from=1-6, to=2-5]
	\arrow[from=1-6, to=2-7]
	\arrow[from=1-10, to=2-8]
	\arrow[from=1-10, to=2-12]
	\arrow[from=2-6, to=1-6]
	\arrow[from=2-6, to=3-6]
	\arrow[from=2-9, to=1-10]
	\arrow[from=2-9, to=2-11]
	\arrow[from=2-9, to=3-10]
	\arrow[from=2-11, to=1-10]
	\arrow[from=2-11, to=3-10]
	\arrow[from=3-6, to=2-5]
	\arrow[from=3-6, to=2-7]
	\arrow[from=3-10, to=2-8]
	\arrow[from=3-10, to=2-12]
\end{tikzcd}\]
In \cite[Theorem 4.2.6]{Stef}, Stefanich showed that $\Spanth(\cK_0)$ has a universal property (recalled in \cref{rec-span}), which we use to prove our main result.

\begin{theorem}[{\cref{main-thm}}]\label{main-thm-intro}
	The functor above extends to a $3$-functor
	\[
		\tCatK^{(-)}\colon \Spanth(\cK_0) \too \thCatt.
	\]
\end{theorem}

This indeed subsumes the two previous results.
\cref{lim-colim-intro} follows by noting that the spans $\pt \gets X = X$ and $X = X \to \pt$ are adjoint to one another from both sides in $\Spanth(\cK_0)$.
Moreover, our previous work on iterated spans shows that the equivalence is given by a certain norm map \cite[Proposition 2.24]{shifted}.
Similarly, \cref{constant-intro} follows by applying the $3$-functor to the endomorphisms of $\pt \in \Spanth(\cK_0)$.

\begin{remark}\label{lim-dual}
	Our results admit a dual version for the $2$-category $\tCatKlim$ of categories admitting $\cK$-indexed \emph{limits} and functors preserving them, through the $2$-equivalence
	\[
		(-)^\op\colon \tCatK \iso (\tCatKlim)^{\co}.
	\]
	In particular, \cref{main-thm-intro} gives a $3$-functor
	\[
		(\tCatKlim)^{(-)}\colon \Spanth(\cK_0)^{\trop} \too \thCatt.
	\] 
\end{remark}

\begin{remark}\label{connection-to-span2}
	As was mentioned above, in our previous work we have shown that the $2$-category of iterated spans $\Spant(\cK_0)$, that is, the maximal $2$-category in $\Spanth(\cK_0)$, carries a $\cK_0$-semiadditve $\cK_0$-parameterized object \cite[Proposition 2.24]{shifted}.
	Moreover, we conjecture that it is the universal example when $\cK_0$ is truncated \cite[Conjecture 2.25]{shifted}.
	In light of this, one may wonder why \cref{main-thm-intro} provides a more highly structured result.
	This stems from the categorical nature of $\tCatK$, and has a decategorified analogue.
	Harpaz proved that $\Spano(\Spacesmfin)$ is the free $m$-semiadditive category, so that any object $X \in \DD$ in an $m$-semiadditive category lifts to a functor $X^{(-)}\colon \Spano(\Spacesmfin) \to \DD$, encoding the distributivity of integration over restriction.
	In the categorical case $\DD = \Catmfin$, the integration is left adjoint to the restriction, and their distributivity data is the Beck--Chevalley map.
	Thus, in this case, the universal property of the $2$-category of spans $\Spanoh(\Spacesmfin)$ (recalled in \cref{rec-span}) shows that every object $\CC \in \Catmfin$ further lifts to a $2$-functor $\CC^{(-)}\colon \Spanoh(\Spacesmfin) \to \tCatmfin$.
\end{remark}


\subsection{Sketch of the Proof}

We briefly outline the proof of the results above for the reader's convenience.

\begin{enumerate}
	\item \cref{constant-intro} for $\CC = \SpacesK$:
	We prove the case of constant diagrams on $\SpacesK$ by comparing both sides to the functor $\PShK\colon \Cat \to \CatK$, the left adjoint to the inclusion.
	The comparison to constant colimits is essentially immediate by preservation of colimits (see \cref{colim-PK}).
	The comparison to constant limits is established by analyzing the full subcategory $\PShK(X) \subset \PSh(X) \simeq \Spaces^X$ using the Yoneda density theorem (see \cref{lim-PK}).

	\item \cref{constant-intro} for arbitrary $\CC \in \CatK$:
	From the first step, we deduce that $\SpacesK[X]$ is self-dual, and that the duality sends $\SpacesK[X] \to \SpacesK[Y]$ to its right adjoint (see \cref{dual-R}).
	This implies the case of an arbitrary constant diagram by tensoring with $\CC$ (see \cref{constant}).
	This step closely follows the proof for the case of $\PrL$ from \cite[\S4.1]{CatChar}.

	\item \cref{main-thm-intro}:
	We prove that the functor in question satisfies the property for which $\Spanth(\cK_0)$ is universal, by reducing it to the case of a constant diagram from the last step, together with certain coherence conditions verified using the universal property of $\Spanoh(\cK_0)$ (see \cref{main-thm}).
	
	\item \cref{lim-colim-intro}:
	We deduce the identification of limits and colimits by noting that the two adjoints of the span $\pt \gets X = X$ coincide in the $3$-category $\Spanth(\cK_0)$, which is preserved under the $3$-functor from \cref{main-thm-intro} (see \cref{lim-colim}).
\end{enumerate}

\subsection{Acknowledgements}

I thank the anonymous referee for carefully reading the manuscript and for their helpful comments and suggestions.
I am grateful to Lior Yanovski for numerous conversations about spans and ambidexterity over the years.
I thank Bastiaan Cnossen and Thomas Blom for useful discussions, and Shai Keidar for his comments on the manuscript.
I would like to thank the Max Planck--Weizmann joint postdoctoral program for its financial support, and the Max Planck Institute for Mathematics in Bonn for its hospitality during the workshop on ``Unstable Homotopy Theory'', during which a part of this manuscript was written.

\section{The Proof}

We fix a full subcategory of spaces $\mdef{\cK_0} \subset \Spaces$ including the point and closed under pullbacks (or empty), and a collection of categories containing it $\cK_0 \subset \mdef{\cK} \subset \Cat$.
We denote by $\mdef{\CatK}$ the category of categories admitting $\cK$-indexed colimits and functors preserving them.
Recall from \cite[\S4.8.1 and Lemma 4.8.4.2]{HA} that it is presentably symmetric monoidal, with unit denoted by $\mdef{\SpacesK}$.
Moreover, the inclusion $\CatK \subset \Cat$ admits a symmetric monoidal left adjoint
\[
	\mdef{\PShK}\colon \Cat \too \CatK.
\]
This makes $\CatK$ into a commutative algebra, and in particular a module, over $\Cat$ in $\PrL$.
Recalling from \cite[Theorem 1.2]{Hei} that presentable $2$-categories are equivalent to modules over $\Cat$ in $\PrL$, this makes $\CatK$ into a $2$-category, which we denote by $\mdef{\tCatK}$.

\subsection{Higher Span Categories}\label{rec-span}

In this paper we shall make use of the $2$-category of spans and the $3$-category of iterated spans, featuring in \cref{main-thm-intro}.
We begin by recalling these categories and their universal properties, and we fill some gaps missing from the literature.
For this section, we fix the following notation.

\begin{notn}
	We let $\cX$ denote a category admitting pullbacks (which we will later take to be $\cK_0$).
\end{notn}

The $2$-category $\mdef{\Spanoh(\cX)}$, was constructed by Haugseng \cite{spans} (where it is denoted by $\Spano^+(\cX)$), to which we refer the reader for a detailed discussion.
We informally recall that this is the $2$-category whose objects are the objects of $\cX$, a $1$-morphism from $X$ to $Y$ is a span $X \gets Z \to Y$ (and composition is given by pullback), and a $2$-morphism is a map between spans.
This $2$-category enjoys a universal property, originally stated by Gaitsgory--Rozenblyum \cite[Chapter 7, Theorem 3.2.2]{GR}, and recently proven independently by Macpherson \cite{Macph} (where it is denoted by $\mrm{Corr}_{\cX}$) and Stefanich \cite{Stef} (where it is denoted by $2\mrm{Corr}(\cX)$). We also refer the reader to \cite[Theorem 2.2.7]{EH} for another presentation of Macpherson's proof.

\begin{defn}
	Let $\tA$ be a $2$-category.
	We say that a commutative square
	\[\begin{tikzcd}
		D & A \\
		B & C
		\arrow["{f'}", from=1-1, to=1-2]
		\arrow["{L'}"', from=1-1, to=2-1]
		\arrow["L", from=1-2, to=2-2]
		\arrow["f"', from=2-1, to=2-2]
	\end{tikzcd}\]
	is (vertically) \tdef{right adjointable} if $L$ and $L'$ admit right adjoints $R$ and $R'$, and the Beck--Chevalley map is an isomorphism
	\[
		f' R' \iso R f.
	\]
\end{defn}

\begin{defn}
	We say that a functor $F\colon \cX \to \tA$ to a $2$-category $\tA$ satisfies the \tdef{left Beck--Chevalley condition} if it takes every pullback square in $\cX$ to a right adjointable square in $\tA$.
\end{defn}

With this at hand, we can state the universal property of $\Spanoh(\cX)$.

\begin{thm}[{\cite[Theorem 3.4.18]{Stef}}]\label{univ-1.5}
	Let $\tA$ be a $2$-category.
	The inclusion $\cX \to \Spanoh(\cX)$ induces an equivalence
	\[
		\Map(\Spanoh(\cX), \tA) \iso \Map_{\BC^\LL}(\cX, \tA)
	\]
	from the space of $2$-functors $\Spanoh(\cX) \to \tA$ to the space of functors $\cX \to \tA$ satisfying the left Beck--Chevalley condition.
\end{thm}

If $\cX$ is symmetric monoidal, and the tensor product preserves pullbacks, then the $2$-category $\Spanoh(\cX)$ inherits a compatible symmetric monoidal structure.
Macpherson proved a symmetric monoidal version of the universal property.

\begin{thm}[{\cite[4.4.6 Theorem]{Macph}}]\label{univ-1.5-sm}
	Assume that $\cX$ is symmetric monoidal and that the tensor product preserves pullbacks.
	Let $\tA$ be a symmetric monoidal $2$-category.
	The inclusion $\cX \to \Spanoh(\cX)$ induces an equivalence
	\[
		\Map^\otimes(\Spanoh(\cX), \tA) \iso \Map^\otimes_{\BC^\LL}(\cX, \tA)
	\]
	from the space of symmetric monoidal $2$-functors $\Spanoh(\cX) \to \tA$ to the space of symmetric monoidal functors $\cX \to \tA$ satisfying the left Beck--Chevalley condition.
\end{thm}

The $3$-category of iterated spans $\mdef{\Spanth(\cX)}$ mentioned in the introduction was also constructed by Haugseng \cite{spans} (where it is denoted by $\Spant^+(\cX)$).
Stefanich proved an analogues universal property for this $3$-category (which he denotes by $3\mrm{Corr}(\cX)$), which we now recall.

\begin{defn}
	Let $\trA$ be a $3$-category.
	We say that a commutative square
	\[\begin{tikzcd}
		D & A \\
		B & C
		\arrow["{f'}", from=1-1, to=1-2]
		\arrow["{L'}"', from=1-1, to=2-1]
		\arrow["L", from=1-2, to=2-2]
		\arrow["f"', from=2-1, to=2-2]
	\end{tikzcd}\]
	is (vertically) \tdef{$2$-fold right adjointable} if the following conditions hold:
	\begin{enumerate}
		\item $L, L', f, f'$ admit right adjoints $R, R', g, g'$, with counit maps $c_L, c_{L'}, c_f, c_{f'}$.
		\item The square is right adjointable.
		\item The commutative square
		\[\begin{tikzcd}
			{Lf'g'R \simeq fL'R'g} & fg \\
			LR & {\Id_C}
			\arrow["{c_{L'}}", from=1-1, to=1-2]
			\arrow["{c_{f'}}"', from=1-1, to=2-1]
			\arrow["{c_f}", from=1-2, to=2-2]
			\arrow["{c_L}"', from=2-1, to=2-2]
		\end{tikzcd}\]
		in the $2$-category $\tEnd(C)$ is right adjointable.
	\end{enumerate}
\end{defn}

\begin{defn}
	We say that a functor $F\colon \cX \to \trA$ to a $3$-category $\trA$ satisfies the \tdef{$2$-fold left Beck--Chevalley condition} if it takes every pullback square in $\cX$ to a $2$-fold right adjointable square in $\trA$.
\end{defn}

\begin{thm}[{\cite[Theorem 4.2.6]{Stef}}]\label{univ-2.5}
	Let $\trA$ be a $3$-category.
	The inclusion $\cX \to \Spanth(\cX)$ induces an equivalence
	\[
		\Map(\Spanth(\cX), \trA) \iso \Map_{\BC_2^\LL}(\cX, \trA)
	\]
	from the space of $3$-functors $\Spanth(\cX) \to \trA$ to the space of functors $\cX \to \trA$ satisfying the $2$-fold left Beck--Chevalley condition.
\end{thm}

Important for our applications will be the case where the target $2$-category is $\tA = \tCat$, the $2$-category of categories, or, more generally, $\tA = \tCatK$.
Recall that the $2$-category of spans encodes two functorialities adjoint to one another (and satisfying the Beck--Chevalley condition).
At the same time, one can pass between left and right adjoints via the mate equivalence
\[
	\tCat^\LL \simeq \tCat^{\RR,\co,\op}
\]
from \cite[3.1.11]{laxoplax} (see also \cite[Corollary 4.7.4.18(3)]{HA}).
Recall that a morphism in $\tCatK$ is a left adjoint, if and only if the underlying functor is a left adjoint (in $\tCat$) and its right adjoint preserves $\cK$-indexed colimits (i.e.\ belongs to $\tCatK$), and the mate equivalence restricts to the corresponding subcategories
\[
	\tCatK^\LL \simeq \tCatK^{\RR,\co,\op}.
\]
As one might expect, the two constructions are compatible with one another, as we now prove.

\begin{notn}\label{mate-equiv}
	Let $\CC_\bullet\colon I \to \tCatK$ be a functor from a category $I$.
	Assume that it sends each morphism to a left adjoint morphism, i.e.\ lands in $\tCatK^\LL$.
	We denote by $\mdef{\CC_\bullet^\RR}\colon I^\op \to \tCatK$ the diagram obtained by passing to right adjoints via the mate equivalence above.
	Dually, we denote by $\mdef{\CC_\bullet^\LL}$ the diagram obtained by passing to left adjoints.
\end{notn}

\begin{prop}\label{mate-span}
	Let $\CC_\bullet\colon \cX \to \tCatK$ be a functor satisfying the left Beck--Chevalley condition.
	Then $\CC_\bullet^\RR$ is equivalent to the pre-composition of the lift $\overline{\CC}_\bullet\colon \Spanoh(\cX) \to \tCatK$ from \cref{univ-1.5} with the inclusion $\cX^\op \to \Spanoh(\cX)$.
\end{prop}

\begin{proof}
	We first handle the case where $\cK = \emptyset$, i.e.\ for a functor landing in $\tCat$.
	Recall that the passage between $\CC_\bullet$ and $\CC_\bullet^\RR$ using the mate equivalence from \cite[3.1.11]{laxoplax} (or \cite[Corollary 4.7.4.18(3)]{HA}) is given by first applying cocartesian unstraightening, observing that the resulting cocartesian fibration is also cartesian, and applying cartesian straightening.
	Now, \cite[Theorem 3.12]{UnivUnf} gives a description of the inverses of the equivalences
	\begin{align*}
		\Map(\Spanoh(\cX), \tCat) &\iso \Map_{\BC^\LL}(\cX, \tCat), \\
		\Map(\Spanoh(\cX), \tCat) &\iso \Map_{\BC^\RR}(\cX^\op, \tCat)
	\end{align*}
	from \cref{univ-1.5} and its dual.
	These inverses are given by applying (co)cartesian unstraightening, applying the functor $\Spanoh(-)$ to the resulting fibration, and finally applying cartesian straightening.
	This shows that the two constructions in question are indeed equivalent.

	The result follows formally for arbitrary $\cK$ as we now explain.
	Let $\iota\colon \tCatK \hookrightarrow \tCat$ denote the inclusion, and consider the composition $\iota\CC_\bullet\colon \cX \to \tCat$.
	By the previous case, we know that $(\iota\CC_\bullet)^\RR$ is equivalent to the pre-composition of the lift $\overline{\iota\CC}_\bullet$ with $\cX^\op \to \Spanoh(\cX)$.
	Now, the uniqueness of lifts shows that there is an equivalence $\overline{\iota\CC}_\bullet \simeq \iota\overline{\CC}_\bullet$ of functors $\Spanoh(\cX) \to \tCat$.
	Since the mate equivalence for $\tCatK$ is defined by restricting the one for $\tCat$, we have an equivalence $(\iota\CC_\bullet)^\RR \simeq \iota\CC_\bullet^\RR$.
	Combining these, we get that $\iota\CC_\bullet^\RR$ is equivalent to the pre-composition of the lift $\iota\overline{\CC}_\bullet$ with $\cX^\op \to \Spanoh(\cX)$, which concludes the proof since $\iota$ is an inclusion.
\end{proof}

While we expect the mate equivalence to be natural in $2$-functors, we are unaware of such a result in the literature.
Nevertheless, we can use the last proposition to prove that, if we further assume the Beck--Chevalley condition, as follows.

\begin{prop}\label{mate-functorial}
	Let $\cK, \cK' \subset \Cat$ be two collections of categories and let $F\colon \tCatK \to \tCatKp$ be any $2$-functor.
	Let $\CC_\bullet\colon \cX \to \tCatK$ be a functor satisfying the left Beck--Chevalley condition.
	There is a canonical equivalence of functors $\cX^\op \to \tCatKp$
	\[
		(F\CC_\bullet)^\RR \simeq F\CC_\bullet^\RR.
	\]
\end{prop}

\begin{proof}
	The uniqueness of lifts shows that there is an equivalence $\overline{F\CC}_\bullet \simeq F\overline{\CC}_\bullet$ of functors $\Spanoh(\cX) \to \tCatKp$.
	Therefore, their restrictions along $\cX^\op \to \Spanoh(\cX)$, which by \cref{mate-span} are $(F\CC_\bullet)^\RR$ and $F\CC_\bullet^\RR$, are also equivalent.
\end{proof}

\cref{mate-span} shows that the $2$-category of spans implements the passage between left and right adjoints functorially.
For our applications, it will be important to understand the following specific case in more detail.
Let $\CC$ be a cocomplete category, and consider the functor
\[
	\CC^{(-)}\colon \Spaces^\op \too \tCat
\]
sending $X \in \Spaces$ to $\CC^X$ and $f\colon X \to Y$ to the pre-composition functor $f^*\colon \CC^Y \to \CC^X$.
Since $\CC$ is cocomplete, the functors $f^*$ admit left adjoints $f_!$ given by left Kan extension, and it follows from \cite[Proposition 4.3.3]{HL} that $\CC^{(-)}$ satisfies the left Beck--Chevalley condition.
Therefore, by \cref{univ-1.5}, it extends to a $2$-functor
\[
	\CC^{(-)}\colon \Spanoh(\Spaces) \too \tCat.
\]
We know that this sends the span $\pt \gets X \to \pt$ to the functor $f_{X!}f_X^*\colon \CC \to \CC$, namely, the constant colimit functor $\colim_X(-)$.
Our next goal is to prove this description functorially in $X$.
Recall that the category of endomorphisms of the point is $\End_{\Spanoh(\Spaces)}(\pt) \simeq \Spaces$.
Thus, the lift $\CC^{(-)}$ induces a functor
\[
	F\colon \Spaces \too \End(\CC),
\]
sending $X$ to the image of the span $\pt \gets X \to \pt$ under $\CC^{(-)}$, i.e.\ to $\colim_X(-)$.
We claim that $F$ agrees with the colimit functor coherently.
Namely, that $F$ is equivalent to $\Id_\CC[-]\colon \Spaces \to \End(\CC)$, given by the left Kan extension of $\Id_\CC\colon \pt \to \End(\CC)$ along the inclusion $\pt \to \Spaces$.
We prove this in two steps, starting with the following weaker lemma.

\begin{lem}\label{F-Id-pw}
	The functors $F$ and $\Id_\CC[-]$ above agree point-wise on objects, morphisms and $2$-morphisms.
	That is, for $X \in \Spaces$ we have
	\[
		F(X) \simeq f_{X!}f_X^* \simeq \colim_X(-).
	\]
	Under this identification, for $f\colon X \to Y$, the map $F(f)$ identifies with
	\[
		f_{X!}f_X^* \simeq f_{Y!}f_!f^*f_Y^* \too[c_f] f_{Y!}f^*.
	\]
	Finally, under this identification, for a homotopy $\alpha\colon f \simeq g$, the $2$-morphism $F(f) \simeq F(g)$ identifies with the $2$-morphism $c_f \simeq c_g$.
\end{lem}

\begin{proof}
	Recall from \cite[4.4.17]{CohAdj} that for a fixed functor $R\colon \EE \to \DD$, the space of counits, that is, pairs of $L\colon \DD \to \EE$ and $c\colon LR \to \Id_\EE$ exhibiting $L$ as left adjoint to $R$, is either contractible (if $R$ is a right adjoint) or empty (if it is not).
	We apply this to our situation.
	Let $f\colon X \to Y$ be a map of spaces, and recall that $\CC^{(-)}\colon \Spaces^\op \to \tCat$ sends $f$ to $f^*\colon \CC^Y \to \CC^X$.
	Also recall that the $2$-morphism in $\Spanoh(\Spaces)$
	\[\begin{tikzcd}[row sep=scriptsize]
		& X \\
		Y && Y \\
		& Y
		\arrow["f"', from=1-2, to=2-1]
		\arrow["f", from=1-2, to=2-3]
		\arrow["f"', from=1-2, to=3-2]
		\arrow[equals, from=2-1, to=3-2]
		\arrow[equals, from=2-3, to=3-2]
	\end{tikzcd}\]
	is a counit exhibiting the spans $Y \gets X = X$ and $X = X \to Y$ as adjoints (see for example \cite[Proposition 3.3.1]{Stef}).
	Using the uniqueness result above, and that the extension
	\[
		\CC^{(-)}\colon \Spanoh(\Spaces) \too \tCat
	\]
	is a $2$-functor, we see that the spans $Y \gets X = X$ and $X = X \to Y$ are sent to the functors $f^*$ and the left Kan extension $f_!$, and the $2$-morphism above is sent to their counit
	\[
		c_f\colon f_!f^* \too \Id_{\CC^Y}.
	\]
	With this in mind we return to the statement of the lemma.

	For an object $X \in \Spaces$, the functor $F(X)$ is given by the image of the span $\pt \gets X \to \pt$, corresponding to the case $f_X\colon X \to \pt$ above, so we get a canonical equivalence
	\[
		F(X) \simeq f_{X!}f_X^* \simeq \colim_X(-).
	\]

	For a morphism $f\colon X \to Y$, the morphism $F(f)$ is the image of the $2$-morphism in $\Spanoh(\Spaces)$ depicted on the left, which is equivalent to the composition of $2$-morphisms on the right
	\[\begin{tikzcd}[column sep=small,row sep=small]
		& X &&&& Y && X && Y \\
		\pt && \pt & \simeq & \pt && Y && Y && \pt \\
		& Y &&&& Y && Y && Y
		\arrow[from=1-2, to=2-1]
		\arrow[from=1-2, to=2-3]
		\arrow["f"', from=1-2, to=3-2]
		\arrow[from=1-6, to=2-5]
		\arrow[equals, from=1-6, to=2-7]
		\arrow[equals, from=1-6, to=3-6]
		\arrow["f"', from=1-8, to=2-7]
		\arrow["f", from=1-8, to=2-9]
		\arrow["f"', from=1-8, to=3-8]
		\arrow[from=1-10, to=2-11]
		\arrow[equals, from=1-10, to=3-10]
		\arrow[equals, from=2-7, to=3-8]
		\arrow[equals, from=2-9, to=1-10]
		\arrow[equals, from=2-9, to=3-10]
		\arrow[from=3-2, to=2-1]
		\arrow[from=3-2, to=2-3]
		\arrow[from=3-6, to=2-5]
		\arrow[equals, from=3-6, to=2-7]
		\arrow[equals, from=3-8, to=2-9]
		\arrow[from=3-10, to=2-11]
	\end{tikzcd}\]
	Thus, using the description of the counit from the discussion above, we conclude that $F(f)$ is equivalent to
	\[
		f_{X!}f_X^* \simeq f_{Y!}f_!f^*f_Y^* \too[c_f] f_{Y!}f_Y^*.
	\]

	Finally, for $\alpha\colon f \simeq g$, the uniqueness of counits implies that the $2$-morphism $F(f) \simeq F(g)$ is identified with the isomorphism $c_f \simeq c_g$.
\end{proof}

\begin{prop}\label{F-Id}
	Let $\CC$ be a cocomplete category, and consider the lift
	\[
		\CC^{(-)}\colon \Spanoh(\Spaces) \too \tCat.
	\]
	The functor induced on the endomorphisms of the point
	\[
		F\colon \Spaces \too \End(\CC)
	\]
	is naturally isomorphic to $\Id_\CC[-]$.
\end{prop}

\begin{proof}
	By \cref{F-Id-pw}, $F$ agrees with $\Id_\CC[-]$ on objects, and in particular on the point $F(\pt) \simeq \Id_\CC$.
	Since $\Id_\CC[-]$ is left Kan extended from the point $\pt \in \Spaces$, there is a canonical map $\Id_\CC[-] \to F$.
	It thus suffices to show that $F$ commutes with colimits, which we prove by showing that it preserves (arbitrary) coproducts and pushouts.

	Observe that for coproducts and for pushouts, the cocone diagram which we need to show is universal involves only objects, $1$-morphisms, and (in the case of a pushout) a single $2$-morphism.
	Hence by the non-coherent comparison from \cref{F-Id-pw}, $F$ and $\Id_\CC[-]$ agree on these cocone diagrams, and since $\Id_\CC[-]$ commutes with colimits, this cocone is indeed universal, as required.
\end{proof}

\subsection{The Constant Case}

Recall that the category $\CatK$ is presentable, and in particular both cocomplete and complete, which allows us to define the following functors.

\begin{defn}\label{colim-lim-def}
	Let $\CC \in \tCatK$.
	We define the functors
	\[
		\mdef{\CC[-]}\colon \Spaces \too \tCatK,
		\qquad \mdef{\CC^{(-)}}\colon \Spaces^\op \too \tCatK,
	\]
	sending a space $X$ to the colimit or limit of the constant $X$-indexed diagram on $\CC$, respectively.
\end{defn}

Note that $\CC^{(-)}$ sends $f\colon X \to Y$ to the pre-composition functor $f^*\colon \CC^Y \to \CC^X$.
When the spaces are assumed to be in $\cK_0$, we can say more.

\begin{prop}\label{adjointable}
	For any $\CC \in \tCatK$, the functor
	\[
		\CC^{(-)}\colon \cK_0 \too \tCatK
	\]
	sends morphisms to right adjoints, and sends every pullback square to a left adjointable square.
\end{prop}

\begin{proof}
	Since $\CC$ admits $\cK_0$-indexed colimits, for a map $f\colon X \to Y$ in $\cK_0$, the functor $f^*\colon \CC^Y \to \CC^X$ admits a left adjoint $f_!\colon \CC^X \to \CC^Y$ given by left Kan extension.
	The second part follows from the formula for left Kan extension, as in the proof of \cite[Proposition 4.3.3]{HL}.
\end{proof}

Our first goal, achieved in \cref{constant}, is to compare $\CC[-]$ with $\CC^{(-),\LL}$, where $(-)^\LL$ denotes the passage to left adjoints from \cref{mate-equiv}.
We begin by analyzing the case $\CC = \SpacesK$, where we can say more, by comparing both functors to $\PShK$.

\begin{lem}\label{colim-PK}
	There is a natural equivalence
	\[
		\SpacesK[-] \simeq \PShK(-)
	\]
	of functors $\Spaces \to \tCatK$.
\end{lem}

\begin{proof}
	The two functors have the same value on $\pt$ and commute with colimits, hence they are equivalent, since $\Spaces$ is freely generated from $\pt$ under colimits.
\end{proof}

To prove the second comparison, we shall make use of the twisted arrow category, introduced by Joyal in the higher categorical setting, and extensively studied by Barwick \cite{BarQ,Unfurling} and Lurie \cite[\S5.2.1]{HA}.
We recall that it can be defined as a complete Segal space in the following way.

\begin{defn}[{\cite[Definition A.2.2]{Haugseng_Melani_Safronov_2022}}]
	Let $\CC$ be a category, with associated complete Segal space $\CC_\bullet\colon \Delta^\op \to \Spaces$.
	The \tdef{twisted arrow category} $\mdef{\Tw(\CC)}$ is the category associated with the complete Segal space
	\[
		\Delta^\op \too[{[n] \mapsto [n]^\op \star [n]}] \Delta^\op \too[\CC_\bullet] \Spaces.
	\]
	There is a map $\Tw(\CC) \to \CC$, corresponding to the inclusion $[n] \to [n]^\op \star [n]$.
\end{defn}

\begin{lem}\label{tw-space}
	For any space $X$, the map $\Tw(X) \to X$ is an isomorphism.
\end{lem}

\begin{proof}
	Recall that a category $\CC$ is a space if and only if the associated complete Segal space $\CC_\bullet\colon \Delta^\op \to \Spaces$ is a constant functor.
	We see that the complete Segal space associated with $\Tw(X)$ is constant on $X$ as well, concluding the proof.
\end{proof}

We turn to comparing $\PShK$ to the functor $\SpacesK^{(-),\LL}$.

\begin{lem}\label{lim-PK}
	There is a natural equivalence
	\[
		\PShK(-) \simeq \SpacesK^{(-),\LL}
	\]
	of functors $\cK_0 \to \tCatK$. 
\end{lem}

\begin{proof}
	Recall that the free cocompletion functor
	\[
		\PSh\colon \Cat \too \CAT
	\]
	sends a category $\CC$ to the category of presheaves $\Fun(\CC^\op, \Spaces)$, and the functoriality is given by left Kan extension.
	In other words, we have a natural equivalence
	\[
		\PSh(-) \simeq \Fun((-)^\op, \Spaces)^\LL
	\]
	of functors $\Cat \to \CAT$ (see for example \cite[\S8]{HHLN}).
	Restricting this to $\cK_0 \subset \Cat$, we obtain a natural equivalence
	\[
		\PSh(-) \simeq \Spaces^{(-),\LL}.
	\]
	Recall that in \cite[Proposition 5.3.6.2]{HTT}, Lurie constructs $\PShK$ as a subfunctor of $\PSh$.
	Specifically, given $\CC \in \Cat$, the category $\PShK(\CC)$ is the full subcategory of $\PSh(\CC)$ generated from the image of the Yoneda $\yon\colon \CC \to \PSh(\CC)$ under $\cK$-indexed colimits.
	It suffices to show that under the equivalence above, this subfunctor corresponds to $\SpacesK^{(-),\LL}$.
	This can be check on objects, i.e.\ we need to show that $\PShK(X) \subset \PSh(X)$ corresponds to $\SpacesK^X \subset \Spaces^X$ for any $X \in \cK_0$.

	We start with the inclusion $\PShK(X) \subset \SpacesK^X$.
	For a point $x \in X$, the Yoneda $\yon(x)\colon X \to \Spaces$ is the functor sending $y \in X$ to the path space from $x$ to $y$ in $X$.
	Since $\cK_0$ is closed under pullbacks, this pullback is in $\cK_0$, and in particular $\yon(x) \in \SpacesK^X$.
	Since $\SpacesK$ is closed under $\cK$-indexed colimits, we see that $\PShK(X)$ is indeed contained in $\SpacesK^X$.

	It remains to show the reverse inclusion.
	Recall that the Yoneda density theorem (c.f.\ \cite[Remark 4.2]{coend}) says that any presheaf $F \in \PSh(\CC)$ can be expressed as a weighted colimit (i.e., coend) of representables
	\[
		F \simeq \coliminl^F \yon.
	\]
	That is, $F$ is the colimit of the composition
	\[
		\Tw(\CC) \too \CC^\op \times \CC \too[F \times \yon] \Spaces \times \PSh(\CC) \too[\times] \PSh(\CC).
	\]
	We note that since the functor $\PShK$ is symmetric monoidal, $\PShK(X)$ is a module over $\PShK(\pt) \simeq \SpacesK$, i.e.\ the subcategory $\PShK(X) \subset \PSh(X)$ is closed under taking the product with an object from $\SpacesK \subset \Spaces$.
	Going back to our case, let $F \in \PSh(X)$ be a presheaf taking values in $\SpacesK$.
	By \cref{tw-space}, we have $\Tw(X) \simeq X$, so in this case the colimit above is indexed over $X$.
	The value of the diagram at $x \in X$ is $F(x) \times \yon(x)$.
	Since $\yon(x) \in \PShK(X)$ and $F(x) \in \SpacesK$, we get that $F(x) \times \yon(x) \in \PShK(X)$.
	By the closure of $\PShK(X)$ under $\cK$-indexed colimits, we see that $F$ is in $\PShK(X)$.
	
	In conclusion, we have shown that $\PShK(X) \simeq \SpacesK^X$ under the equivalence constructed above, as required.
\end{proof}

Using these lemmas we conclude our analysis of the case $\CC = \SpacesK$.

\begin{prop}\label{dual-R}
	For any $X \in \cK_0$, the category $\SpacesK[X] \in \CatK$ is self-dual.
	Moreover, we have equivalences
	\[
		\SpacesK[-] \simeq \PShK(-) \simeq \SpacesK^{(-),\LL} \simeq \SpacesK[-]^{\vee,\LL}
	\]
	of functors $\cK_0 \to \tCatK$.
\end{prop}

\begin{proof}
	By \cref{colim-PK} the functor $\SpacesK[-]$ is equivalent to $\PShK(-)$, and in particular it inherits a symmetric monoidal structure.
	By \cref{lim-PK}, it is furthermore equivalent to $\SpacesK^{(-),\LL}$, whence by \cref{adjointable} it sends morphisms to left adjoints and pullback squares to right adjointable squares.
	Thus, by \cref{univ-1.5-sm}, $\SpacesK[-]$ extends into a symmetric monoidal $2$-functor
	\[
		\SpacesK[-]\colon \Spanoh(\cK_0) \too \tCatK.
	\]
	By \cite[Theorem 1.4]{spans}, every object $X \in \Spanoh(\cK_0)$ is self-dual, so we conclude that $\SpacesK[X] \in \tCatK$ is indeed self-dual.

	Turning our attention to the equivalences, note that the first is \cref{colim-PK} and the second is \cref{lim-PK}.
	As for the third equivalence, since the symmetric monoidal structure on $\tCatK$ is closed, we have
	\begin{align*}
		\SpacesK[-]^\vee
		&\simeq \Fun_\cK(\SpacesK[-], \SpacesK)\\
		&\simeq \Fun_\cK(\PShK(-), \SpacesK)\\
		&\simeq \Fun(-, \SpacesK)\\
		&\simeq \SpacesK^{(-)},
	\end{align*}
	where we have used that $\PShK$ is left adjoint to the inclusion $\CatK \subset \Cat$.
\end{proof}

Finally, we prove the case of constant colimits.

\begin{prop}\label{constant}
	There is an equivalence
	\[
		\CC[-] \simeq \CC^{(-),\LL}
	\]
	of functors $\cK_0 \to \tCatK$.
\end{prop}

\begin{proof}
	First, note that as the symmetric monoidal structure on $\tCatK$ preserves colimits in each coordinate, we have
	\[
		\CC[-] \simeq \SpacesK[-] \otimes_\cK \CC.
	\]
	Second, observe that the functor
	\[
		- \otimes_\cK \CC\colon \CatK \too \CatK
	\]
	is colimit preserving and $\CatK$-linear.
	Via the symmetric monoidal functor $\PShK\colon \Cat \to \CatK$, it is also $\Cat$-linear, that is, a morphism in $\Mod_{\Cat}(\PrL)$.
	Thus, using \cite[Theorem 1.2]{Hei} we conclude that it enhances into a $2$-functor
	\[
		- \otimes_\cK \CC\colon \tCatK \too \tCatK.
	\]
	We get the following chain of equivalences
	\begin{align*}
		\CC[-]^\RR
		&\simeq (\SpacesK[-] \otimes_\cK \CC)^\RR\\
		&\simeq \SpacesK[-]^\RR \otimes_\cK \CC\\
		&\simeq \SpacesK[-]^\vee \otimes_\cK \CC\\
		&\simeq \Fun_\cK(\SpacesK[-], \CC)\\
		&\simeq \Fun_\cK(\PShK(-), \CC)\\
		&\simeq \Fun(-, \CC)\\
		&\simeq \CC^{(-)},
	\end{align*}
	where the first step follows from equivalence above, the second step follows from \cref{mate-functorial} using the $2$-functoriality of $- \otimes_\cK \CC$, the third is \cref{dual-R}, the fourth is by duality, the fifth is by \cref{colim-PK}, and the sixth is by the fact that $\PShK$ is left adjoint to the inclusion $\CatK \subset \Cat$.
	The claim follows by passing to left adjoints.
\end{proof}

\subsection{Lifting to Iterated Spans}

Our next goal is to lift the functor
\[
	\tCatK^{(-)}\colon \cK_0^\op \too \thCatt,
\]
sending $X$ to $\tCatK^X$ and $f\colon X \to Y$ to $f^*\colon \tCatK^Y \to \tCatK^X$, to a $3$-functor from $\Spanth(\cK_0)$.
Recall from \cref{adjointable} that $\tCatK^{(-)}$ satisfies the Beck--Chevalley condition, so \cref{univ-1.5} already provides a lift
\[
	\tCatK^{(-)}\colon \Spanoh(\cK_0) \too \thCatt,
\]
which we use in the proof of the following lemma.

\begin{lem}\label{cf-adj}
	Let $f\colon X \to Y$ be a map in $\cK_0$.
	The counit map
	\[
		c_f\colon f_! f^* \too \Id_{\tCatK^Y}
	\]
	is a left adjoint in the $2$-category $\tEnd(\tCatK^Y)$.
\end{lem}

\begin{proof}
	The exponential adjunction gives an equivalence of $2$-categories
	\[
		\tEnd(\tCatK^Y) \simeq \tFun(Y \times \tCatK^Y, \tCatK).
	\]
	Therefore, by \cite[Corollary 4.2.12]{fun-adj}, $c_f$ is a left adjoint if and only if
	\begin{enumerate}
		\item for every $y \in Y$ and $\CC_\bullet\colon Y \to \tCatK$, the map
		\[
			c_f\colon (f_! f^* \CC_\bullet)_y \too \CC_y
		\]
		is a left adjoint, and
		\item for every $y \to z$ in $Y$ and $F_\bullet\colon \CC_\bullet \to \DD_\bullet$ in $\tCatK^Y$, the following commutative square is right adjointable
		\[\begin{tikzcd}
			{(f_! f^* \CC_\bullet)_y} & {(f_! f^* \DD_\bullet)_z} \\
			{\CC_y} & {\DD_z}
			\arrow[from=1-1, to=1-2]
			\arrow["{c_f}"', from=1-1, to=2-1]
			\arrow["{c_f}", from=1-2, to=2-2]
			\arrow[from=2-1, to=2-2]
		\end{tikzcd}\]
	\end{enumerate}
	
	For the first part, we begin by identifying the map in question.
	Let $X_y := X \times_Y \pt$ be the fiber of $X$ over $y$, and denote by $f_y\colon X_y \to \pt$ the unique map.
	Recall that in the $2$-category $\Spanoh(\cK_0)$, the counit map $c_f$ is the $2$-morphism
	\[\begin{tikzcd}[row sep=scriptsize]
		& X \\
		Y && Y \\
		& Y
		\arrow["f"', from=1-2, to=2-1]
		\arrow["f", from=1-2, to=2-3]
		\arrow["f"', from=1-2, to=3-2]
		\arrow[equals, from=2-1, to=3-2]
		\arrow[equals, from=2-3, to=3-2]
	\end{tikzcd}\]
	The evaluation of $c_f$ at $y$ amounts to pulling-back along $y\colon \pt \to Y$, that is, considering the composition of maps of spans appearing on the left in the following diagram, which is in turn equivalent to the composition appearing on the right.
	\[\begin{tikzcd}[column sep=small,row sep=small]
		& X && \pt &&&& \pt && {X_y} \\
		Y && Y && \pt & \simeq & Y && \pt && \pt \\
		& Y && \pt &&&& \pt && \pt
		\arrow["f"', from=1-2, to=2-1]
		\arrow["f", from=1-2, to=2-3]
		\arrow["f"', from=1-2, to=3-2]
		\arrow["y"', from=1-4, to=2-3]
		\arrow[equals, from=1-4, to=2-5]
		\arrow[equals, from=1-4, to=3-4]
		\arrow["y"', from=1-8, to=2-7]
		\arrow[equals, from=1-8, to=2-9]
		\arrow[equals, from=1-8, to=3-8]
		\arrow["{f_y}"', from=1-10, to=2-9]
		\arrow["{f_y}", from=1-10, to=2-11]
		\arrow["{f_y}"', from=1-10, to=3-10]
		\arrow[equals, from=2-1, to=3-2]
		\arrow[equals, from=2-3, to=3-2]
		\arrow[equals, from=2-5, to=3-4]
		\arrow[equals, from=2-9, to=3-10]
		\arrow[equals, from=2-11, to=3-10]
		\arrow["y", from=3-4, to=2-3]
		\arrow["y", from=3-8, to=2-7]
		\arrow[equals, from=3-8, to=2-9]
	\end{tikzcd}\]
	This shows that the map
	\[
		c_f\colon (f_! f^* \CC_\bullet)_y \too \CC_y
	\]
	is equivalent to the map
	\[
		\qquad c_{f_y}\colon f_{y!} f_y^* \CC_y \too \CC_y.
	\]
	By \cref{F-Id}, the latter is equivalent to the map
	\[
		\CC_y[X_y] \too \CC_y.
	\]
	By \cref{constant}, this is equivalent to the colimit functor
	\[
		\colim_{X_y}\colon \CC_y^{X_y} \too \CC_y,
	\]
	which is left adjoint to the constant diagram functor, concluding the first part.
	
	For the second part, note that since $Y$ is a space, the morphism $y \to z$ is an isomorphism, so we may assume that $y = z$.
	As in the first part, \cref{F-Id} and \cref{constant} identify the square in question with
	\[\begin{tikzcd}
		{\CC_y^{X_y}} & {\DD_y^{X_y}} \\
		{\CC_y} & {\DD_y}
		\arrow["{F_y}", from=1-1, to=1-2]
		\arrow["{\colim_{X_y}}"', from=1-1, to=2-1]
		\arrow["{\colim_{X_y}}", from=1-2, to=2-2]
		\arrow["{F_y}"', from=2-1, to=2-2]
	\end{tikzcd}\]
	It is clear that this is right adjointable, as the constant diagram functor commutes with applying $F_y$, concluding the proof.
\end{proof}

We now prove our main theorem.

\begin{thm}\label{main-thm}
	The functor
	\[
		\tCatK^{(-)}\colon \cK_0^\op \too \thCatt,
	\]
	extends to a $3$-functor
	\[
		\tCatK^{(-)}\colon \Spanth(\cK_0) \too \thCatt.
	\]
\end{thm}

\begin{proof}
	To avoid opposites and stay closer to Stefanich's formulation of \cref{univ-2.5}, we work with the left adjoint functoriality.
	More precisely, as explained above, we already have a lift
	\[
		\tCatK^{(-)}\colon \Spanoh(\cK_0) \too \thCatt,
	\]
	and we wish to further lift it to $\Spanth(\cK_0)$, so it suffices to check that its restriction along $\cK_0 \to \Spanoh(\cK_0)$ satisfies the $2$-fold left Beck--Chevalley condition.
	While we do not have a coherent description of this restriction similar to \cref{mate-span}, the description of adjunction data in the $2$-category of spans $\Spanoh(\cK_0)$ (see for example \cite[Proposition 3.3.1]{Stef}) shows that it is given point-wise by passing to left adjoints.
	That is, it sends an object $X \in \cK_0$ to $\tCatK^X$, a morphism $f\colon X \to Y$ to the left Kan extension $f_!\colon \tCatK^X \to \tCatK^Y$, and a $2$-morphism $f \simeq g$ to the canonical $2$-morphism $f_! \simeq g_!$.

	With this in mind, we check the $2$-fold left Beck--Chevalley condition, so let
	\[\begin{tikzcd}
		P & Y \\
		X & Z
		\arrow["{f'}", from=1-1, to=1-2]
		\arrow["{g'}"', from=1-1, to=2-1]
		\arrow["\lrcorner"{anchor=center, pos=0.125}, draw=none, from=1-1, to=2-2]
		\arrow["g", from=1-2, to=2-2]
		\arrow["f"', from=2-1, to=2-2]
	\end{tikzcd}\]
	be a pullback square in $\cK_0$, and let us denote by $h\colon P \to Z$ the composite $f g' \simeq g f'$.
	The functor in question sends this to the commutative square
	\[\begin{tikzcd}
		{\tCatK^P} & {\tCatK^Y} \\
		{\tCatK^X} & {\tCatK^Z}
		\arrow["{f'_!}", from=1-1, to=1-2]
		\arrow["{g'_!}"', from=1-1, to=2-1]
		\arrow["{g_!}", from=1-2, to=2-2]
		\arrow["{f_!}"', from=2-1, to=2-2]
	\end{tikzcd}\]
	which we need to show is $2$-fold right adjointable.
	The morphisms in the diagram admit right adjoints $f^*, f'^*, g^*, g'^*$, and the square is right adjointable by \cite[Proposition 4.3.3]{HL}.
	It remains to show that the commutative square
	\[\begin{tikzcd}
		{h_!h^*} & {f_!f^*} \\
		{g_!g^*} & {\Id_{\tCatK^Z}}
		\arrow["{c_{g'}}", from=1-1, to=1-2]
		\arrow["{c_{f'}}"', from=1-1, to=2-1]
		\arrow["{c_f}", from=1-2, to=2-2]
		\arrow["{c_g}"', from=2-1, to=2-2]
	\end{tikzcd}\]
	in $\tEnd(\tCatK^Z)$ is right adjointable.
	The fact that the vertical morphisms are left adjoints is \cref{cf-adj}, and it remains to check the Beck--Chevalley condition.

	Since the Beck--Chevalley morphism is a natural transformation, we may check that it an isomorphism on each component, i.e.\ evaluate at some $\CC_\bullet \in \tCatK^Z$.
	By the same logic, we may further evaluate at some $z \in Z$.
	Thus, we need to check the Beck--Chevalley condition for the square
	\[\begin{tikzcd}
		{(h_!h^*\CC_\bullet)_z} & {(f_!f^*\CC_\bullet)_z} \\
		{(g_!g^*\CC_\bullet)_z} & {\CC_z}
		\arrow["{c_{g'}}", from=1-1, to=1-2]
		\arrow["{c_{f'}}"', from=1-1, to=2-1]
		\arrow["{c_f}", from=1-2, to=2-2]
		\arrow["{c_g}"', from=2-1, to=2-2]
	\end{tikzcd}\]
	Let $X_z := X \times_Z \pt$ denote the fiber of $X$ over $z$, and similarly for $Y_z$ and $P_z$, and observe that $P_z \simeq X_z \times Y_z$.
	Arguing as in the proof of \cref{cf-adj} using the lift to $\Spanoh(\cK_0)$, and applying \cref{F-Id} and \cref{constant}, we see that the square is equivalent to the square
	\[\begin{tikzcd}
		{\CC_z^{X_z \times Y_z}} & {\CC_z^{X_z}} \\
		{\CC_z^{Y_z}} & {\CC_z}
		\arrow[from=1-1, to=1-2]
		\arrow[from=1-1, to=2-1]
		\arrow[from=1-2, to=2-2]
		\arrow[from=2-1, to=2-2]
	\end{tikzcd}\]
	where all morphisms are given by left Kan extension.
	This square indeed satisfies the Beck--Chevalley condition as explained in \cref{adjointable}, concluding the proof.
\end{proof}

\subsection{Consequences}

We now turn to the proof of \cref{lim-colim-intro}, showing that $\cK_0$-shaped limits and colimits coincide in $\tCatK$.
In fact, the isomorphism between them can be described explicitly via the twisted norm map, whose construction we briefly recall.
We refer the reader to \cite[Definition 2.15]{shifted} for the definition appearing here, and to \cite{TwA} for a more detailed recent account of norm maps.

\begin{defn}\label{Nm-def}
	Let $\DD$ be a category with all small limits and colimits, and let $f\colon X \to Y$ be a map of spaces.
	Consider the commutative diagram
    \[\begin{tikzcd}
        X \\
        & {X \times_Y X} & X \\
        & X & Y
        \arrow["\Delta", from=1-1, to=2-2]
        \arrow[curve={height=-18pt}, Rightarrow, no head, from=1-1, to=2-3]
        \arrow[curve={height=18pt}, Rightarrow, no head, from=1-1, to=3-2]
        \arrow["{\pi_2}", from=2-2, to=2-3]
        \arrow["{\pi_1}"', from=2-2, to=3-2]
        \arrow["\lrcorner"{anchor=center, pos=0.125}, draw=none, from=2-2, to=3-3]
        \arrow["f", from=2-3, to=3-3]
        \arrow["f"', from=3-2, to=3-3]
    \end{tikzcd}\]
    We define the \tdef{dualizing map} to be
	\[
		\mdef{D_f} := \pi_{1*} \Delta_!\colon \DD^X \too \DD^X.
	\]
    We define the twisted \tdef{norm map}
	\[
		\mdef{\Nm_f}\colon f_! D_f \too f_*
	\]
	between functors $\DD^X \to \DD^Y$ to be the composite
	\[
		f_! D_f
		:= f_! \pi_{1*} \Delta_!
		\too[u_f] f_* f^* f_! \pi_{1*} \Delta_!
		\iso f_* \pi_{2!} \pi_1^* \pi_{1*} \Delta_!
		\too[c_{\pi_1}] f_* \pi_{2!} \Delta_!
		\simeq f_*
	\]
	of the unit, Beck--Chevalley, and counit maps.

	Specializing to the case of the terminal map $X \to \pt$, the twisted norm map takes the form
	\[
		\Nm_X\colon \colim_X(D_X(-)) \too \lim_X(-).
	\]
\end{defn}

With these in place, we go back to the case $\DD = \tCatK$.

\begin{prop}\label{lim-colim}
	Let $X \in \cK_0$.
	The dualizing map is canonically isomorphic to the identity $D_X \simeq \Id_{\tCatK^X}$, and the norm map induces an equivalence
	\[
		\Nm_X\colon \colim_X \CC_\bullet \iso \lim_X \CC_\bullet
		\qin \tCatK,
	\]
	natural in diagrams $\CC_\bullet\colon X \to \tCatK$.
\end{prop}

\begin{proof}
	The $3$-functor from \cref{main-thm} sends the span $\pt \gets X = X$ to the constant diagram functor
	\[
		\tCatK \too \tCatK^X.
	\]
	Recall that in the $3$-category $\Spanth(\cK_0)$, this span is both left and right adjoint to the span $X = X \to \pt$.
	We thus conclude that the two adjoints of the constant diagram functor are naturally equivalent, namely that the colimit and limit functors coincide.
	Moreover, \cite[Proposition 2.24]{shifted} shows that the dualizing map is the identity, and that the equivalence is given by the norm map.
\end{proof}

As an example, this reproves the case of presentable categories, only by knowing that the inclusion $\PrL \subset \CAT$ preserves limits (i.e.\ without knowledge about colimits).

\begin{prop}\label{prl}
	Let $X$ be a (small) space.
	The dualizing map is canonically isomorphic to the identity $D_X \simeq \Id_{(\PrL)^X}$, and the norm map induces an equivalence
	\[
		\Nm_X\colon \colim_X \CC_\bullet \iso \lim_X \CC_\bullet
		\qin \PrL
	\]
	natural in diagrams $\CC_\bullet\colon X \to \PrL$.
\end{prop}

\begin{proof}
	Consider the diagram as landing in the larger category $\PrL \subset \CATall$ of all cocomplete categories and colimit preserving functors.
	By applying \cref{lim-colim} to the large categories version with $\cK_0 = \Spaces$ and $\cK = \Cat$, the corresponding statement holds in $\CATall$.
	It thus suffices to show that the inclusion $\PrL \subset \CATall$ preserves these limits and colimits. 

	For the limit, recall that the inclusion $\PrL \subset \CAT$ preserves limits by \cite[Proposition 5.5.3.13]{HTT}, and so does the inclusion $\CATall \subset \CAT$ by \cite[Corollary 5.3.6.10]{HTT}.
	Thus, the inclusion $\PrL \subset \CATall$ preserves limits.

	Now, as the colimit $\colim_X \CC_\bullet$ computed in $\CATall$ is equivalent to the limit $\lim_X \CC_\bullet$, we learn that it lands in $\PrL$.
	Therefore, since $\PrL \subset \CATall$ is a full subcategory, the same cocone exhibits it as the colimit in $\PrL$, as required.
\end{proof}

	\bibliographystyle{alpha}
	\bibliography{refs}

\end{document}